\documentclass[11pt]{article}
\usepackage{amssymb,amsmath,bbm,latexsym,amscd}
\usepackage[all]{xy}
\newtheorem{theorem}[equation]{Theorem}

\newtheorem{lemma}[equation]{Lemma}
\newtheorem{proposition}[equation]{Proposition}
\newtheorem{definition}[equation]{Definition}
\newtheorem{remark}[equation]{Remark}
\numberwithin{equation}{section}

\newcommand{\ot}{\otimes}

\newcommand{\Z}{\mathbb{Z}}
\newcommand{\Cx}{\mathbb{C}}

\newcommand{\fg}{{\mathfrak g}}
\newcommand{\fh}{{\mathfrak h}}

\newcommand{\Max}{\mathrm{Max}}

\newcommand{\ga}{\alpha}

\newcommand{\gs}{\sigma}

\newcommand{\proof}{{\bf Proof\ \ }}
\newcommand{\qed}{\hfill $\Box$}

\newcommand{\cL}{{\mathcal L}}

\newcommand{\ev}{\mathrm{ev}}

\newcommand{\la}{\langle}
\newcommand{\ra}{\rangle}

\title{Weight modules for current algebras}

\author{Daniel Britten$\hbox{}^{1 *}$, Michael Lau$\hbox{}^{2 *}$, and Frank Lemire$\hbox{\,}^{1}$\thanks{Funding from the Natural Sciences and
Engineering Research Council of Canada is gratefully acknowledged.} \vspace{0.3cm}
\vspace{0.1cm}\\$\hbox{\ \,}^1${\small University of Windsor,
Department of Mathematics and Statistics},\\ {\small Windsor, ON, Canada N9B 3P4}\\ \vspace{0.1cm}\\
$\hbox{\ \,}^2${\small Universit\'e Laval,
D\'epartement de math\'ematiques et de statistique},\\ {\small Qu\'ebec, QC, Canada G1V 0A6}\\}
\date{}

\begin{document}

\maketitle

\begin{small}
\noindent {\bf Abstract:} For any finite-dimensional simple Lie algebra $\fg$ and commutative associative algebra $S$ of finite type, we give a complete classification of the simple weight modules of $\fg\ot S$ with bounded weight multiplicities.

\bigskip


\bigskip


\end{small}
\maketitle

\section{Introduction}\label{s1}  The classification of simple weight modules is widely seen as a difficult problem.  In a spectacular tour-de-force, Olivier Mathieu gave a complete description of such modules for finite-dimensional simple Lie algebras \cite{Ma00}.  The crucial part of Mathieu's classification was understanding the simple admissible highest weight modules.

In the present paper, we consider this question in the context of (generalised) current algebras $\cL=\fg\ot_k S$, where $\fg$ is a finite-dimensional simple Lie algebra over $k$, and $S$ is a commutative, associative, and unital $k$-algebra which is of finite type.  These algebras include the finite-dimensional simple Lie algebras $\fg$ and the loop algebras $\fg\ot_{\mathbb{C}}\mathbb{C}[t,t^{-1}]$ of affine Kac-Moody theory.  Our methods work not only for the classification of simple admissible highest weight modules, but also describe the simple objects in the much larger category of {\em admissible modules}, weight modules whose weight multiplicities are uniformly bounded.  We hope that this paper will contribute to solving the harder problem of classifying all simple weight modules for these algebras.

The category of admissible modules includes the finite-dimensional simple modules, all of which are known to be evaluation modules.  (See the introductions of \cite{multiloop,CFK} or the more general results in \cite{NSS,repforms,twcurr} for details.)  What is surprising is that this larger category of generally infinite-dimensional modules can also be described in terms of evaluations.  We now summarize the content of this paper.

In Section 2, we recall basic definitions and note that any classification of simple admissible modules for a current algebra $\cL$ is also a classification of simple admissible modules for its universal central extension $\widetilde{\cL}$, since the central elements must act trivially on any simple weight module. 

Section 3 opens with a review of elementary facts about evaluation representations in the infinite-dimensional context of weight modules.  In particular, we show that a representation $\phi:\ \cL\rightarrow \hbox{End}\,V$ is an evaluation representation if and only if it factors through a direct sum of several copies of $\fg$.  Key parts of this result have been used in past work (for example, in \cite{Ma00}), but we are not aware of a proof appearing in the literature.  We then address the natural question of which evaluation modules are admissible.  The answer (Theorem \ref{3.14}) is strikingly simple: an evaluation module is admissible if and only if each of its tensor components (corresponding to the action of each copy of $\fg$) is admissible and at most one component is infinite dimensional.  The proof, contained in a series of propositions and lemmas, is a detailed study of the relevant structure and combinatorics of indecomposable root systems, building on work of Benkart, Fernando, and two of the authors of the present paper \cite{fernando, BrLe,BBL}.

In Section 4, we classify admissible modules for current algebras.  We use the finite dimensionality constraint on the weight spaces to prove that the kernel of any irreducible weight representation is cofinite in $\cL$.  It is thus of the form $\fg\ot I$ for some cofinite ideal $I$ of $S$.  We then use admissibility to show that $I$ is a radical ideal.  

While these two results are obvious in the context of finite-dimensional modules, they are much more challenging in the potentially infinite-dimensional context of weight representations, and completely new arguments are required.  Cofiniteness is proved using the cofiniteness of the space $\{s\in S\ |\ (x\ot s)v=0\}$ for each $v\in V$ and root vector $x\in \fg$.  The idea is then to express $I$ as a finite intersection of such spaces.  To prove that $I$ is a radical ideal, we note that if it is not, then there is a nonzero ideal $N\subset S/I$ for which $N^2=0$.  We then use $\mathfrak{sl}_2$-combinatorics together with admissibility to argue that there is a nonzero vector $w\in V$ annihilated by the ideal $\fg\ot N\subseteq\cL$.  The space killed by $\fg\ot N$ is a nonzero submodule of $V$, so is equal to $V$ since $V$ is simple.  This contradicts the fact that $\fg\ot S/I$ must act faithfully on $V$. 

Once we know that $I$ is a cofinite radical ideal of $S$, it is straightforward to prove Theorem \ref{modules-are-evaluation}, that every simple admissible module of $\cL$ is an evaluation module.  We then apply Theorem \ref{3.14} to obtain Theorem \ref{classification}, the main result of this paper, a complete description of the isomorphism classes of simple admissible modules in terms of certain finitely supported maps from $\Max\,S$ to the moduli space of admissible modules for $\fg$.

\bigskip

\noindent
{\bf Acknowledgements.}  Much of this work was completed while M.L. was on sabbatical at the University of California, Berkeley.  He thanks the U.C. Berkeley mathematics department for its hospitality during his visit.  The authors also wish to thank Georgia Benkart and Vera Serganova for helpful discussions during the preparation of this paper.


\section{Preliminaries}\label{s2}
Throughout this paper, $\fg$ will denote a finite-dimensional simple Lie algebra over an algebraically closed field $k$ of characteristic zero.  All tensor products and algebras will be taken over the base field $k$.  Fix a Cartan subalgebra $\fh$ of $\fg$, let $\Phi$ be the corresponding root system, and let $\Delta\subset\Phi$ be a base of simple roots.  The set of positive and negative roots of $(\fg,\fh,\Delta)$ will be denoted by $\Phi^+$ and $\Phi^-$, respectively, with corresponding triangular decomposition $\fg=\mathfrak{n}_-\oplus\fh\oplus\mathfrak{n}_+$, where $\mathfrak{n}_\pm=\oplus_{\alpha\in\Phi^\pm}\fg_\alpha$.  For each positive root $\ga\in\Phi^+$, fix nonzero elements $x_{\pm\ga}\in\fg_{\pm\ga}$ and $h_\ga\in\fh$, so that $\mathcal{B}=\{x_\ga,h_\beta\ |\ \ga\in\Phi,\beta\in\Delta\}$ is a Chevalley basis of $\fg$ and each $\{x_\ga,x_{-\ga},h_\ga\}$ forms an $\mathfrak{sl}_2$-triple: $[x_\ga,x_{-\ga}]=h_\ga$, $[h_\ga,x_\ga]=2x_\ga$, and $[h_\ga,x_{-\ga}]=-2x_\ga$ for all $\ga\in\Phi^+$.  We write $(x|y)$ for the Killing form of $x,y\in \fg$, and $U(L)$ for the universal enveloping algebra of any Lie algebra $L$.

Let $S$ be a commutative, associative, and unital $k$-algebra of finite type.  In this paper, we consider modules for the {\em (generalised) current algebra} $\cL=\fg\ot S$, with bracket 
 $
[x\otimes r, y\otimes s]= [x,y]\otimes rs,
$
for all $x,y\in\fg$ and $r,s\in S$.  When $S=k$, $\cL$ is the finite-dimensional simple Lie algebra $\fg$; when $k=\Cx$ and $S=\Cx[t,t^{-1}]$, $\cL$ becomes the loop algebra of affine Kac-Moody theory.  Other well-known examples are the classical current algebra $\fg\ot k[t]$, the multiloop algebra $\fg\ot k[t_1^{\pm 1},\ldots,t_N^{\pm 1}]$, and the $3$-point algebra $\fg\ot k[t,t^{-1},(t-1)^{-1}]$.  Viewing the vector space $\fg$ as an affine scheme, the Lie algebra $\cL$ can also be interpreted geometrically as a space of morphisms from $\hbox{Spec}\,S$ to $\fg$, under pointwise Lie bracket.  See \cite{NSS} or \cite{twcurr} for details.

Each representation of $\cL$ is also a representation of any central extension of $\cL$, by letting the central elements act trivially.  There is a well-known construction \cite{Ka85} of the universal central extension $\widetilde{\cL}$ of $\cL$ using the vector space quotient $\la S,S\ra:=(S\ot S)/Q$, where $Q$ is the subspace spanned by the set
$$
\{r\otimes s+s\otimes r,\ rs\otimes t+ st\otimes r+tr\otimes s\ |\  r,s,t\in S\}.
$$
Explicitly, $\widetilde{\cL}$ is isomorphic to, and will henceforth be identified with, the Lie algebra $\left(\fg\otimes  {S}\right)\oplus \la {S},{S}\ra,$ with bracket 
\[
[x \otimes a+\la r,s\ra, y\otimes b+\la u,v\ra]= [x,y]\otimes (ab)+ (x | y)\la a, b\ra,
\]
for all $x,y\in\fg$ and $a,b,r,s,u,v\in S$, where $\la r,s\ra$ denotes the coset $r\ot s+Q$ in $\la{S},{S}\ra$, for all $r,s\in S$.

\begin{definition}{\em  An $\cL$- or $\widetilde{\cL}$-module $V$ is said to be a {\em weight module} if $V$ decomposes as a direct sum of finite-dimensional common eigenspaces under the action of $\fh\ot 1$.  That is,
\[
V=\bigoplus_{\lambda\in \fh^*} V_\lambda,
\]
where each {\em weight space} $V_\lambda=\{v\in V\mid (h\otimes 1)v=\lambda (h)v\hbox{ for all } h\in \fh\}$ is finite dimensional, and the elements $\lambda\in\fh^*$ for which $V_\lambda\neq 0$ are called {\em weights}.  A nonzero vector $v\in V$ is a {\em maximal vector} if $(e \otimes s) v=0$ for all $e\in\mathfrak{n}_+$ and $s\in S$.
A weight module  $V$ is a {\em highest weight module} of weight $\mu$ if $V_\mu$ is $1$-dimensional, and there is a maximal vector $v\in V_\mu$ generating the module $V$.  
A weight module $V=\bigoplus_{\lambda\in \fh^*} V_\lambda$  is said to be {\em admissible} if there exists $B\in \Z_{>0}$ such that $\dim V_\lambda \leq B$ for all $\lambda$. Weight modules for current algebras $L\ot S$ of semisimple Lie algebras $L$ are defined analogously.}
\end{definition}

The main goal of this paper is to give an explicit construction of all simple admissible modules of $\cL$.  The finite-dimensional simple $\cL$-modules were described by Chari and Rao in the case where $S$ is the algebra $k[t,t^{-1}]$ of Laurent polynomials \cite{chari86,rao93}.  Their classification easily extends to finite-dimensional simple $\cL$-modules when $S$ is arbitrary.  See the introduction of \cite{multiloop} or \cite{CFK}, or the more general results in \cite{NSS,repforms,twcurr}, for instance.  By Theorem \ref{centre-acts-as-zero} below, this classification also covers finite-dimensional simple modules for the corresponding universal central extensions $\widetilde{\cL}$.  Naturally, these modules are included in our construction.

The following proposition allows us to restrict our attention to $\cL$-modules.  Our original proof was simplified to the argument below following a discussion with G.~Benkart.

\begin{theorem} \label{centre-acts-as-zero} If $V$ is a simple weight module for $\widetilde{\cL}$, then every central element $\la r,s\ra$ of $\widetilde{\cL}$ acts trivially on $V$.  
\end{theorem}
\begin{proof}  Since $S$ is of finite type, it is countable dimensional over $k$.  Since $\fg$ is finite dimensional and $V$ is simple, it follows that $V$ is countable dimensional, and by Schur's Lemma, the central elements $\la r,s\ra$ of $\widetilde{\cL}$ act as scalars $\lambda_{r,s}$ on $V$.


Since the restriction of the Killing form to the Cartan subalgebra $\fh$ is nondegenerate, there exist $h_1,h_2\in \fh$ with $(h_1|h_2)\neq 0$.  The trace of the operator 
$$(h_1\ot r) \circ (h_2\ot s)-(h_2\ot s) \circ (h_1\ot r)$$
restricted to any (finite-dimensional) weight space $V_\nu$ is zero.  But $[h_1\ot r,h_2\ot s]=(h_1|h_2)\la r,s\ra$ acts as the scalar $(h_1|h_2)\lambda_{r,s}$, and thus has trace $\dim(V_\nu)(h_1|h_2)\lambda_{r,s}$ on $V_\nu$.  Therefore, $\la r,s\ra$ acts as zero on all nonzero weight spaces, and thus acts trivially everywhere on $V$. \qed 
\end{proof}

\section{Evaluation representations}\label{s2.5}
  
We will later show that simple admissible modules can be classified in terms of {\em evaluation representations}.  To explain this statement, let $M_1,\ldots,M_r\subset S$ be maximal ideals.  There is then an {\em evaluation map}
\begin{eqnarray*}
\hbox{ev}_{\underline{M}}:\ \cL&\rightarrow&\fg^{\oplus r}\\
x\ot s&\mapsto&\big(s(M_1)x,\ldots,s(M_r)x\big),\\ 
\end{eqnarray*}
where $s(M_i)\in k$ is the residue of $s$ modulo $M_i$:
$$s(M_i)+M_i=s+M_i,$$
as elements of $S/M_i\cong k$.  By the Chinese Remainder Theorem, $\hbox{ev}_{\underline{M}}$ is surjective whenever $M_1,\ldots, M_r$ are distinct, and any tensor product $W_1\ot \cdots\ot W_r$ of simple $\fg$-modules then pulls back to an irreducible $\cL$-module under the evaluation map:
$$\cL\stackrel{\hbox{ev}_{\underline{M}}}{\longrightarrow}\fg^{\oplus r}\longrightarrow\hbox{End}(W_1\ot \cdots\ot W_r),$$
where the action of $\cL$ on $W_1\ot \cdots\ot W_r$ is given by the formula
$$(x\ot s).(w_1\ot\cdots\ot w_r)=\sum_{i=1}^rs(M_i)w_1\ot\cdots\ot xw_i\ot\cdots\ot w_r.$$
Such representations are called {\em evaluation representations} of $\cL$, and we denote the evaluation representation above by $V(\underline{M},\underline{W})$, where $\underline{M}=(M_1,\ldots,M_r)$ and $\underline{W}=(W_1,\ldots,W_r)$.  
If $W_1,\ldots,W_r$ are weight modules for $\fg$ and $M_1,\ldots,M_r$ are distinct maximal ideals of $S$, then we say that $V(\underline{M},\underline{W})$ is an {\em evaluation weight module}.  

\begin{remark}\label{infinite-mults} 
Note that evaluation weight modules may have infinite weight multiplicities, so they are not necessarily weight modules for $\cL$.  For example, consider the infinite-dimensional $\mathfrak{sl}_2(\mathbb{C})$-module $W$ spanned by the linearly independent set $\{v_i\ |\ i\in\mathbb{Z}\}$, where
$$hv_i=2iv_i,\ ev_i=-\frac{(2i+1)^2}{4}v_{i+1},\ \hbox{and}\ fv_i=v_{i-1},$$
and $[h,e]=2e,\ [h,f]=-2f,$ and $[e,f]=h$.  The module $W$ is a simple weight module for $\mathfrak{sl}_2(\mathbb{C})$, yet the evaluation weight module $W\ot W$ obtained by evaluating $\cL=\mathfrak{sl}_2(\mathbb{C})\ot S$ at any two distinct maximal ideals of $S$ has infinite-dimensional weight spaces.
\end{remark}

Our next goal is to prove the following theorem, which says that a given simple weight $\cL$-module is an evaluation weight module if the $\cL$-action $\phi:\ \cL\rightarrow \hbox{End}\,V$ factors through a direct sum of finitely many copies of $\fg$.

\begin{theorem}\label{thm:evaluations}
Let $V$ be a simple weight module for $\cL$, with module action given by $\phi:\ \cL\rightarrow \hbox{\em End}\,V$.  Suppose there exist pairwise distinct $M_1,\ldots ,M_r\in\hbox{\em Max}\,S$ and a Lie algebra homomorphism $\psi:\ \fg^{\oplus r}\rightarrow\hbox{\em End}\,V$ such that $\phi=\psi\circ\hbox{ev}_{\underline{M}}$.  Then $V$ is isomorphic to an evaluation weight module of $\cL$.
\end{theorem}

We will use the following lemma and proposition to prove this theorem.

\begin{lemma}\label{lm:submodules}
Let $H$ be a Cartan subalgebra of a semisimple Lie algebra $L$, and let $V$ be a weight module for $L$. Suppose that there exists $\lambda\in H^*$ such that the space 
$$W_\lambda=\{w\in W\ |\ hw=\lambda(h)w\hbox{\ for all\ }h\in H\}$$ 
is nonzero for all nonzero submodules $W\subseteq V$.  Then $V$ contains a simple $L$-submodule. 
\end{lemma}
\begin{proof}
Let $W\subseteq V$ be a nonzero $L$-submodule for which $\dim W_\lambda$ is minimal.  We claim that $U(L)W_\lambda\subseteq W$ is a simple submodule.

Indeed, for any nonzero $m\in U(L)W_\lambda$, we see that the $\lambda$-weight space $(U(L)m)_\lambda$ of the submodule $U(L)m\subseteq W$ is equal to $W_\lambda$, by the minimality of $\dim W_\lambda$.  Thus $U(L)m$ contains $W_\lambda$, so $U(L)m=U(L)W_\lambda$, and $U(L)W_\lambda$ is a simple $L$-module.\qed
\end{proof}

\begin{proposition}\label{prop:simples of sums}
Let $\fg_1$ and $\fg_2$ be finite-dimensional semisimple Lie algebras, with Cartan subalgebras $\fh_1$ and $\fh_2$, respectively.  Suppose that $V$ is a simple weight module for $\fg_1\oplus \fg_2$, with respect to the Cartan subalgebra $\fh_1\oplus \fh_2$.  Then $V\cong V_1\ot V_2$ for some simple weight modules $V_i$ for $(\fg_i,\fh_i)$.
\end{proposition}
\begin{proof}
Let $v\in V$ be a nonzero vector of weight $(\lambda,\mu)\in \fh_1^*\times\fh_2^*$.  Then $W=U(\fg_1)v\subseteq V$ is clearly a weight module for $(\fg_1,\fh_1)$, since $W_\eta$ is contained in the (finite-dimensional) weight space $V_{\eta,\mu}$ for all $\eta\in\fh_1^*$.

Let $M\subseteq W$ be any nonzero $\fg_1$-submodule.  Then $\hbox{Hom}_{\fg_1}(M,V)$ has the structure of a nonzero weight module over $\fg_2$, with action $(x\rho)(m)=x(\rho(m)),$ for all $x\in\fg_2$, $\rho\in\hbox{Hom}_{\fg_1}(M,V)$, and $m\in M$.  Let $N\subseteq \hbox{Hom}_{\fg_1}(M,V)$ be a nonzero $\fg_2$-submodule.  The map 
\begin{eqnarray*}
\Psi_{M,N}:\ M\ot N&\longrightarrow&V\\
m\ot\rho&\longmapsto&\rho(m)
\end{eqnarray*}
is a nonzero $\fg_1\oplus\fg_2$-module homomorphism, so is surjective by the simplicity of $V$.  In particular, $\Psi_{M,N}$ restricts to a surjection
$M_\lambda\ot N_\mu\longrightarrow V_{\lambda,\mu},$ where $M_\lambda$, $N_\mu$, and $V_{\lambda,\mu}$ are the weight spaces of weights $\lambda\in\fh_1^*$, $\mu\in\fh_2^*$, and $(\lambda,\mu)\in\fh_1^*\times\fh_2^*$, respectively.  Thus $M_\lambda$ and $N_\mu$ are nonzero (since $V_{\lambda,\mu}\neq 0$).

By Lemma \ref{lm:submodules}, there exists a nonzero simple (weight) $\fg_1$-submodule $A\subseteq W$ and a nonzero simple (weight) $\fg_2$-submodule $B$ of $\hbox{Hom}_{\fg_1}(A,V)$.  It is completely straightforward to verify that $A\ot B$ is a simple $\fg_1\oplus\fg_2$-module, so the surjective map
$$\Psi_{A,B}:\ A\ot B\longrightarrow V$$
is a $\fg_1\oplus\fg_2$-module isomorphism.\qed
\end{proof}

\bigskip

\noindent
{\bf Proof of Theorem \ref{thm:evaluations}}  Using the Chinese Remainder Theorem, choose elements $s_1,\ldots ,s_r$ of $S$ so that $s_i(M_j)=\delta_{ij}$ for all $i,j$.  Let $\mathcal{H}$ denote the finite-dimensional abelian Lie subalgebra $\fh\ot \hbox{Span}\,\{s_1,\ldots,s_r\}$ of $\cL$.  Since $\mathcal{H}$ commutes with $\fh\ot 1$, its action preserves the (finite-dimensional) weight spaces of $V$.  Let $V_\lambda$ be a (nonzero) weight space of $V$.  By Lie's theorem (or elementary linear algebra), $\mathcal{H}$ has a nonzero common eigenvector $v\in V_\lambda$.

Note that the evaluation map $\hbox{ev}_{\underline{M}}$ restricts to a vector space isomorphism 
$$\fg\ot\hbox{Span}\,\{s_i\}\rightarrow\fg_i:=0\oplus\cdots\oplus\fg\oplus\cdots\oplus 0,$$
where $\fg$ is the $i$th summand of the expression on the right.  Thus $v$ is a common eigenvector for the action of the Cartan subalgebra $\hbox{ev}_{\underline{M}}(\mathcal{H})=\fh^{\oplus r}$ of $\fg^{\oplus r}$.  The $\fg^{\oplus r}$-module $U(\fg^{\oplus r})v$ is a nonzero submodule of the simple $\fg^{\oplus r}$-module $V$, so $V=U(\fg^{\oplus r})v$, and $V$ is a simple weight module for $\fg^{\oplus r}$, with respect to the Cartan subalgebra $\fh^{\oplus r}$.  The result now follows from Proposition \ref{prop:simples of sums} by induction on $r$.\qed

\bigskip

\begin{remark} Conversely, by the surjectivity of the map $\hbox{ev}_{\underline{M}}$, it is clear that an evaluation module $V(\underline{M},\underline{W})$ is simple if $\underline{M}=(M_1,\ldots,M_r)$ and $\underline{W}=(W_1,\ldots,W_r)$, where $M_1,\ldots,M_r$ are distinct maximal ideals of $S$, and $W_1,\ldots,W_r$ are simple $\fg$-modules.
\end{remark}

It is also relatively straightforward to give an explicit isomorphism criterion for evaluation weight modules.  Let $\mathcal{M}$ be the set of all isomorphism classes of simple weight $(\fg,\fh)$-modules.  The modules in $\mathcal{M}$ can be described as cuspidals or simple quotients of parabolically induced modules \cite{fernando}.  As has been done in finite-dimensional contexts by many authors, evaluation weight representations can be labelled by functions $\Psi:\ \hbox{Max}\,S\rightarrow\mathcal{M}$.  In particular, the evaluation module $V(\underline{M},\underline{W})$ is labelled by the function $\Psi_{\underline{M},\underline{W}}$, where 
$$\Psi_{\underline{M},\underline{W}}(N)=\left\{\begin{array}{ll}
[W_i] & \hbox{\ if\ }N=M_i\\
\lbrack k \rbrack & \hbox{\ otherwise},
\end{array}\right.$$
where $\underline{M}=(M_1,\ldots,M_r)$, and $[k]$ is the isomorphism class of the trivial one-dimensional $\fg$-module.  The support $\hbox{supp}\,\Psi$ of such functions is always of finite cardinality, where
$$\hbox{supp}\,\Psi=\{M\in\hbox{Max}\,S\ | \Psi(M)\neq[k]\}.$$
The correspondence sending a simple evaluation weight module $V(\underline{M},\underline{W})$ to a finitely supported function $\Psi_{\underline{M},\underline{W}}:\ \Max\,S\rightarrow\mathcal{M}$ is clearly surjective, and we now verify that it descends to a well-defined surjective map on the level of isomorphism classes of simple evaluation weight modules.

Let $\alpha:\ \cL\rightarrow\hbox{End}\,V(\underline{M},\underline{W})$ and $\beta:\ \cL\rightarrow\hbox{End}\,V(\underline{N},\underline{U})$ be isomorphic simple evaluation weight representations of $\cL$, where $\underline{M}=(M_1,\ldots,M_r)$ and $\underline{N}=(N_1,\ldots,N_s)$.  Since $\fg$ is simple, we see that $$\fg\ot\bigcap_{i=1}^rM_i=\ker\alpha=\ker\beta=\fg\ot\bigcap_{j=1}^sN_j.$$
From the Nullstellensatz, it follows that $\{M_1,\ldots,M_r\}=\{N_1,\ldots,N_s\}$, so we can write $N_i=M_{\gs(i)}$ for some permutation $\gs$ of $\{1,\ldots,r\}$.  As in the finite-dimensional case (see \cite{NSS} or \cite{repforms}, for instance), it is now straightforward to verify that $U_i\cong W_{\gs(i)}$ as $(\fg,\fh)$-weight modules, so $\Psi_{\underline{M},\underline{W}}=\Psi_{\underline{N},\underline{U}}$.  There is thus a well-defined map $[V(\underline{M},\underline{W})]\mapsto\Psi_{\underline{M},\underline{W}}$ from isomorphism classes of simple evaluation weight representations to finitely supported functions $\Max\,S\rightarrow\mathcal{M}$.  The surjectivity of this correspondence is clear.  That this map is injective follows from the fact that $\cL$-modules form a symmetric monoidal category: that is, permuting the tensor factors $W_1,\ldots,W_r$ of an evaluation representation $V(\underline{M},\underline{W})$ will not change the isomorphism class of the $\cL$-module, provided we apply the same permutation to the sequence of maximal ideals $M_1,\ldots,M_r$ used in the evaluation.  We have now proved the following theorem:

\begin{theorem}\label{labelling-of-evalreps}
The isomorphism classes of simple evaluation weight $\cL$-modules are in natural bijection with the finitely supported functions $\Psi:\ \Max\,S\rightarrow\mathcal{M}.$  Explicitly, the isomorphism class of $V(\underline{M},\underline{W})$ corresponds to the function $\Psi_{\underline{M},\underline{W}}$.\qed
\end{theorem}

As we have already noted in Remark \ref{infinite-mults}, evaluation weight modules do not always have finite-dimensional weight spaces, let alone satisfy the admissibility condition that the dimensions of these weight spaces be uniformly bounded.  We now turn our attention to the question of which evaluation weight modules are admissible.  The answer is surprisingly simple: an evaluation weight module $V(\underline{M},\underline{W})$ is admissible if and only if each component $W_i$ of $\underline{W}=(W_1,\ldots,W_\ell)$ is an admissible $\fg$-module, and at most one $W_i$ is infinite dimensional.  It is interesting to note that by \cite{BBL}, simple infinite-dimensional admissible modules exist only for $\fg$ of types A and C, so every simple admissible evaluation weight module is finite dimensional if $\fg$ is not of one of these types.

We now review a few results we will need from \cite{fernando}, \cite{BrLe}, and \cite{BBL}.  Let $W$ be a simple weight module for $\fg$.  We write $T=\{\ga\in\Phi\ |\ x_\ga\hbox{\ acts injectively on\ }W\}$ and $N=\{\ga\in\Phi\ |\ x_\ga\hbox{\ acts locally nilpotently on\ }W\}$.  By a lemma of Fernando \cite[Lemma 2.3]{fernando}, the root system $\Phi$ decomposes as $\Phi=T\cup N$, and $T$ is a convex set.  That is, $T=\{\sum_{\ga\in T}c_\ga\ga\ |\ c_\ga\geq 0\}\cap \Phi.$  For each $X\subseteq \Phi$, we write $-X$ for the set $\{-\ga\ |\ \ga\in X\}\subseteq\Phi$, and $X^+$ (respectively, $X^-$) for the positive (resp., negative) roots in $X$, relative to a fixed base of simple roots.  Define 
\begin{equation}\label{TandNsets}
T_s=T\cap(-T),\ T_a=T\setminus T_s,\ N_s=N\cap(-N),\ N_a=N\setminus N_s.
\end{equation}
Clearly, $N_a=-T_a$.  Suppose that $X\subseteq X'$ are subsets of $\Phi$.  If $\ga+\beta\in X$ whenever $\ga\in X$, $\beta\in X'$, and $\ga+\beta\in\Phi$, then we say that $X$ is an {\em ideal} of $X'$.

\begin{proposition}{\bf \cite{BrLe,BBL}} \label{BBL-prop}

(1)\quad $T_s$ and $N_s$ are root subsystems of $\Phi$.

\smallskip

(2)\quad There exists a base $B$ of $\Phi$ such that $N_a\subseteq \Phi_B^+$, where $\Phi_B^+$ (respectively, $\Phi_B^-$) is 

\quad\quad \ the set of positive (resp., negative) roots relative to $B$.

\smallskip

(3)\quad If $B$ is any base such that $N_a\subseteq\Phi_B^+$, then $N_a$ is an ideal of $\Phi_B^+$, $T_a$ is an ideal 

\quad\quad \ of $\Phi_B^-$, and $B\cap T_s$ is a base of simple roots for the root subsystem $T_s$.\qed
\end{proposition}

\noindent
The following lemma and its proof are a small generalisation of \cite[Lemma 4.7(iii)]{BBL}.

\begin{lemma}\label{lemma0}
If $N_a\subseteq\Phi_B^+$, $\ga\in N_s$, and $\beta\in T_s$, then $\ga+\beta$ is {\em not} a root.
\end{lemma}
\proof Suppose $\ga+\beta\in\Phi$.  By \cite[Lemma 2.3]{fernando}, $\ga+\beta\in T$ or $\ga+\beta\in N$.  If $\ga+\beta\in T$, then $\ga=(\ga+\beta)+(-\beta)\in T$, since $-\beta\in T_s\subseteq T$ and $T$ is a convex subset of $\Phi$.  This is a contradiction since $\ga\in N$, so we see that $\ga+\beta\in N_s$ or $\ga+\beta\in N_a$.

If $\ga+\beta\in N_s$, then $\beta=(\ga+\beta)+(-\ga)\in N_s$, since $N_s$ is a root subsystem of $\Phi$.  But then $\beta\in N_s\cap T$, a contradiction.  Hence $\ga+\beta\in N_a$.  But then $-(\ga+\beta)\in T_a$ and $-\ga=-(\ga+\beta)+\beta\in T$ by the convexity of $T$.  Thus $\ga\notin N_s$, a contradiction.  Hence $\ga+\beta$ is not a root.\qed

\begin{lemma}\label{lemma1} Suppose that $W_1$ and $W_2$ are simple infinite-dimensional admissible $\fg$-modules with $T_1\cap(\pm T_2)\neq\emptyset$, where $T_1$ and $T_2$ are the sets of roots $\ga$ for which $x_\alpha$ acts injectively on $W_1$ and $W_2$, respectively.  Then $W_1\ot W_2$ is {\em not} an admissible $\fg$-module.
\end{lemma}
\proof Choose a pair of nonzero weight spaces $W_{1,\lambda}\subseteq W_1$ and $W_{2,\mu}\subseteq W_2$.  Suppose that $\ga\in T_1\cap(-T_2)$.  Then the $(\lambda+\mu)$-weight space of $W_1\ot W_2$ contains the infinite-dimensional space
$$\sum_{n=0}^\infty W_{1,\lambda+n\alpha}\ot W_{2,\mu-n\alpha},$$
so $W_1\ot W_2$ is not admissible (nor is it even a weight module).

Now assume that $\ga\in T_1\cap T_2$.  Then the $(\lambda+\mu+n
\alpha)$-weight space of $W_1\ot W_2$ contains the sum
$$\sum_{\ell=0}^n W_{1,\lambda+\ell\ga}\ot W_{2,\mu+(n-\ell)\ga},$$
which has dimension at least $n+1$.  Since this computation is valid for any positive integer $n$, there is no uniform bound on the dimensions of the weight spaces of $W_1\ot W_2$.\qed

\begin{lemma}\label{lemma-2}  Let $\ga$ be a simple root in a base $B$, and let $\langle -\ga\rangle$ be the ideal generated by $-\ga$ in the set $\Phi^-$ of negative roots.  Suppose that $\beta\in\Phi\setminus\pm\langle -\ga\rangle$.  Then there exists $\gamma\in\langle-\ga\rangle$ such that $\beta+\gamma\in\langle-\ga\rangle$.
\end{lemma}
\proof Written as a $\mathbb{Z}$-linear combination of the simple roots $\ga_1,\ldots,\ga_\ell\in B$, the highest root $\theta$ is of the form $\theta=\sum_{i=1}^\ell n_i\ga_i,$ where $n_i\geq 1$ for all $i$.  Since $\fg$ is simple and the Chevalley generators $x_{\ga_i},x_{-\ga_i}$ generate $\fg$, there is a nonzero expression $\hbox{ad}\,z_r\cdots\hbox{ad}\,z_1(x_\alpha)\in\fg_\theta$ for some collection of $z_j\in\{x_{\ga_i},x_{-\ga_i}\ |\ 1\leq i\leq \ell\}$.  Using the Serre relations, it is easy to see that we can choose the $z_j$ to be positive root vectors: $z_j\in\{x_{\ga_i}\ |\ 1\leq i\leq \ell\}$ for all $j$.  That is, there is a sequence
\begin{equation}\label{*}
\ga=\mu_0\prec\mu_1\prec\cdots\prec\mu_r=\theta\hbox{\ (partial order taken with respect to\ }B\hbox{)}, 
\end{equation}
where $\mu_i-\mu_{i-1}\in B$ and $\mu_i\in\Phi$ for all $i\geq 0$, and $\mu_{-1}$ is defined to be zero.

The bilinear form $(-,-):\ \fh^*\times\fh^*\rightarrow k$ induced by Killing form is nondegenerate, and every simple root in $B$ appears as $\mu_i-\mu_{i-1}$ for some $i\geq 0$.  Therefore, $(\beta,\mu_s)$ is nonzero for some $0\leq s\leq r$.  Hence $\beta+\mu_s$ or $\beta-\mu_s$ is a root.

For each element $\tau$ of the root lattice, we write $\tau$ (uniquely) as a $\mathbb{Z}$-linear combination $\tau=\sum_{\gamma\in B} n_\gamma^\tau\gamma$ of the simple roots in $B$.  By the argument used to justify (\ref{*}), it is straightforward to verify that
$$\langle-\ga\rangle=\{-\ga-\mu\in\Phi\ |\ \mu\in\Phi^+\cup\{0\}\}=\{\tau\in\Phi\ |\ n_\ga^\tau<0\}.$$
Thus $n_\ga^\beta=0$, $n_\ga^{\beta+\mu_s}>0$, and $n_\ga^{\beta-\mu_s}<0$.  For any root $\tau$, all the nonzero coefficients $n_\gamma^\tau$ must have the same sign, so we see that if $\beta+\mu_s$ is a root, then $\beta+\mu_s\in\Phi^+$.  Similarly, if $\beta-\mu_s$ is a root, then $\beta-\mu_s\in\Phi^-$.  

If $\beta+\mu_s\in\Phi^+$, then $-\beta-\mu_s\in\Phi$ and $n_\ga^{-\beta-\mu_s}<0$, so $-\beta-\mu_s\in\langle-\ga\rangle$.  Then $\beta+(-\beta-\mu_s)=-\mu_s\in\langle -\ga\rangle$.  By the same argument, if $\beta-\mu_s\in\Phi^-$, then $\beta-\mu_s$ and $-\mu_s$ are both elements of $\langle -\ga\rangle$.\qed

\bigskip

The weight spaces of an $\cL$-module are precisely the weight spaces of its restriction to the Lie subalgebra $\fg\ot 1\subseteq \cL$.  An $\cL$-module is thus admissible if and only if it is admissible as a $\fg$-module.  Restricted to $\fg\ot 1$, evaluation modules are tensor products of $\fg$-modules, so the admissibility of evaluation representations can be settled using the following key proposition.

\begin{proposition}\label{admissible-for-g} Suppose that $W_1$ and $W_2$ are simple infinite-dimensional weight modules for $\fg$.  Then $W_1\ot W_2$ is {\em not} an admissible $\fg$-module.
\end{proposition}
\proof Let $T_i$, $N_i$, $T_{i,s}$, $T_{i,a}$, $N_{i,s}$, and $N_{i,a}$ be the sets (\ref{TandNsets}) associated with the $\fg$-module $W_i$ for $i=1,2$.  Let $B$ be a base of $\Phi$ chosen so that $N_{1,a}\subseteq \Phi_B^+$.  Since $W_1$ is infinite dimensional, $N_{1,s}\neq\Phi$.  Thus $\Phi_B^+=T_{1,s}^+\cup N_{1,s}^+\cup N_{1,a}$.  In particular, $B\subseteq T_{1,s}\cup N_{1,s}\cup N_{1,a}$.

By Proposition \ref{BBL-prop}(3), $B\cap T_{1,s}$ is a base for $T_{1,s}$.  If $B\subseteq T_{1,s}$, then $T_{1,s}=\Phi$.  But then $T_1\cap T_2=\Phi\cap T_2=T_2$.  Since $W_2$ is infinite dimensional and simple, $T_2\neq\emptyset$, so by Lemma \ref{lemma1}, $W_1\ot W_2$ is not admissible.

Without loss of generality, we may therefore assume that $B\not\subseteq T_{1,s}$.  If $B\subseteq T_{1,s}\cup N_{1,s}$, then it follows from Lemma \ref{lemma0} that $\Phi=T_{1,s}\cup N_{1,s}$, since $\ga+\beta$ is never a root if $\ga\in T_{1,s}$ and $\beta\in N_{1,s}$.  In particular, $(\ga,\beta)=0$ for all $\ga\in T_{1,s}$ and $\beta\in N_{1,s}$, so $\Phi$ is decomposable, a contradiction.  Hence $B\not\subseteq T_{1,s}\cup N_{1,s}$, so there is a simple root $\ga$ in $N_{1,a}$.

Since $N_{1,a}$ is an ideal of $\Phi_B^+$, we see that the ideal generated by $\ga$ in $\Phi_B^+$ is contained in $N_{1,a}$.  Since $T_{1,a}=-N_{1,a}$, we have $\langle -\ga\rangle\subseteq T_{1,a}$, in the notation of Lemma \ref{lemma-2}.

By Lemma \ref{lemma1}, we may assume that $T_2\cap(\pm T_1)=\emptyset$.  Moreover, since $W_2$ is infinite dimensional, $T_2$ is nonempty.  Let $\beta\in T_2$.  Clearly, $\beta\notin\pm\langle-\ga\rangle$.  From Lemma \ref{lemma-2}, there exists $\gamma\in\langle -\ga\rangle$ such that $\gamma+\beta\in\langle -\ga\rangle$.

Choose nonzero weight spaces $W_{1,\lambda}\subseteq W_1$ and $W_{2,\mu}\subseteq W_2$.  Then for any nonnegative integer $n$, the $(\lambda+\mu+n(\gamma+\beta))$-weight space of $W_1\ot W_2$ contains the sum
$$\sum_{\ell=0}^n W_{1,\lambda+\ell\gamma+(n-\ell)(\gamma+\beta)}\ot W_{2,\mu+\ell\beta}.$$
Since $\gamma,\gamma+\beta\in T_1$ and $\beta\in T_2$, we see that
$$\dim(W_1\ot W_2)_{\lambda+\mu+n(\gamma+\beta)}>n,$$
so $W_1\ot W_2$ is not admissible.\qed

\bigskip

The main result of this section, the following criterion for admissibility of evaluation weight modules, is now an easy corollary of Proposition \ref{admissible-for-g}.

\begin{theorem}\label{3.14}  Let $V(\underline{M},\underline{W})$ be a simple evaluation weight module for $\fg\ot S$, where $\underline{W}=(W_1,\ldots,W_\ell)$.  Then $V(\underline{M},\underline{W})$ is admissible if and only if each of the $W_i$ is admissible and at most one of the $W_i$ is infinite dimensional.
\end{theorem}
\proof If each of the $W_i$ is admissible and no more than one of them is infinite dimensional, then it is clear that $V(\underline{M},\underline{W})$ is admissible.  The converse follows directly from Proposition \ref{admissible-for-g}.\qed

\section{Simple admissible modules}\label{s3}
While the main goal of this section is to describe the simple admissible modules of $\cL$ (and consequently, of the universal central extension of $\cL$), we impose the admissibility condition only after some general results which hold for arbitrary simple weight modules.

Let $V$ be a simple weight module for $\cL$, with action given by $\phi:\ \cL\rightarrow \hbox{End}\,V$.  Since $\fg$ is finite dimensional and simple, it is easy to verify that the ideals of $\cL$ are all of the form $\fg\ot P$ for ideals $P$ of the commutative (but not necessarily reduced) $k$-algebra $S$.  In particular, $\ker\phi=\fg\ot I$ for some ideal $I=I(V)$ of $S$.  Our goal is to show that $I$ is cofinite in $S$. 

For each $\ga\in\Phi\cup\{0\}$, $v\in V$, and $\lambda\in\fh^*$, let
\begin{eqnarray*}
J_\ga(v)&=&\{s\in S\ |\ (x\ot s)v=0\hbox{\ for all\ }x\in\fg_\ga\},\\
J(v)&=&\bigcap_{\beta\in\Phi\cup\{0\}}J_\beta(v),\\
J(\lambda)&=&\bigcap_{w\in V_\lambda}J(w),\\
J(\lambda,\Phi)&=&\bigcap_{\mu\in\lambda+(\Phi\cup\{0\})}J(\mu).
\end{eqnarray*}
\begin{lemma}
The vector space $J(\lambda,\Phi)$ is of finite codimension in $S$ for each $\lambda\in\fh^*$.
\end{lemma}
\begin{proof}
Since $V_\lambda$ is finite dimensional, $J(\lambda)$ can be expressed as a finite intersection $\cap_i J(v_i)$, where the $v_i$ form a $k$-basis of $V_\lambda$.  The space $J(\lambda,\Phi)$ is thus the intersection of {\em finitely many} spaces of the form $J_\ga(v)$.

Let $x\in\fg_\ga$ and $v\in V_\lambda$.  There is then a linear map 
\begin{eqnarray*}
\eta_{x,v}:\ S&\longrightarrow&V_{\lambda+\ga}\\
s&\longmapsto&(x\ot s)v,
\end{eqnarray*}
which obviously descends to an injection of $S/\ker\eta_{x,v}$ into the finite-dimensional space $V_{\lambda+\ga}$, so $S/\ker\eta_{x,v}$ is finite dimensional.  Fixing a basis $\mathcal{B}_\ga$ of $\fg_\ga$ for $\ga\in\Phi\cup\{0\}$, we see that $J_\ga(v)=\cap_{x\in\mathcal{B}_\ga}\ker\eta_{x,v}$ for all $v\in V_\lambda$, so $J(\lambda,\Phi)$ can be expressed as an intersection taken over some finite collection of homogeneous elements $x\in \fg$ and $v\in V$.  In particular,
\begin{eqnarray*}
S/J(\lambda,\Phi)&\longrightarrow&\bigoplus_{x,v}S/\ker\eta_{x,v}\\
s+J(\lambda,\Phi)&\longmapsto&(s+\ker\eta_{x,v})_{x,v}
\end{eqnarray*} 
is a linear injection.  Since each $S/\ker\eta_{x,v}$ is finite dimensional, we see that $J(\lambda,\Phi)$ is of finite codimension in $S$.\qed
\end{proof}

\bigskip

\begin{lemma}\label{lemma2}
$J(\lambda,\Phi)S\subseteq J(\lambda)$ for each weight $\lambda$.
\end{lemma}
\begin{proof}
Let $r\in J(\lambda,\Phi)$, $s\in S$, and $v\in V_\lambda$.  If $x,y$ are elements of our Chevalley basis $\mathcal{B}$, then 
\begin{align*}
\big([x,y]\ot rs\big)v&=[x\ot r,y\ot s]v\\
&=(x\ot r)(y\ot s)v-(y\ot s)(x\ot r)v.
\end{align*}
Since $y\in\fg_\ga$ for some $\ga\in\Phi\cup\{0\}$ and $v\in V_\lambda$, we see that $(y\ot s)v\in V_{\lambda+\ga}$, so $x\ot r$ annihilates both $(y\ot s)v$ and $v$, by definition.  The Lie algebra $\fg$ is simple, thus perfect, so
$$(\fg\ot rs)v=\big([\fg,\fg]\ot rs\big)v=0,$$
and $rs\in J(\lambda)$.\qed
\end{proof}

\bigskip

\begin{proposition}\label{kernel-is-cofinite}
Let $\lambda$ be a weight of a simple weight $\cL$-module $V$.  Then $I=J(\lambda,\Phi)$, where $I=I(V)$.  In particular, the ring $S/I$ is a finite-dimensional $k$-algebra.
\end{proposition}
\begin{proof}
Let $v$ be a nonzero element of $V_\lambda$ and let $w$ be an arbitrary element of $V$.  Since $V$ is an irreducible representation, $w$ can be expressed as a linear combination of elements of the form $(x_1\ot s_1)\cdots(x_\ell\ot s_\ell)v,$ where $\ell\geq 0$, $s_1,\ldots,s_\ell\in S$, and $x_1,\ldots x_\ell\in\fg$.  

We use induction on $\ell$ to show that the expression
$$(x\ot rs)(x_1\ot s_1)\cdots(x_\ell\ot s_\ell)v$$
is zero for all $x\in \fg$, $r\in J(\lambda,\Phi)$, and $s\in S$.  If $\ell=0$, then by Lemma \ref{lemma2}, $(x\ot rs)v=0$ for all $x\ot rs\in\fg\ot J(\lambda,\Phi)S$.  Otherwise, $\ell\geq 1$ and 
\begin{align*}
(x\ot rs)(x_1\ot s_1)\cdots (x_\ell\ot s_\ell)v=(x_1\ot s_1)&(x\ot rs)(x_2\ot s_2)\cdots (x_\ell\ot s_\ell)v\\
&+\big([x,x_1]\ot rss_1\big)(x_2\ot s_2)\cdots(x_\ell\ot s_\ell)v.
\end{align*}
By induction hypothesis, 
$$(x\ot rs)(x_2\ot s_2)\cdots (x_\ell\ot s_\ell)v=0$$ and $$
\big([x,x_1]\ot rss_1\big)(x_2\ot s_2)\cdots(x_\ell\ot s_\ell)v=0,$$ so $rs\in I$ for all $r\in J(\lambda,\Phi)$ and $s\in S$.  Taking $s=1$, we see that $J(\lambda,\Phi)\subseteq I$.  Since $I$ is obviously contained in $J(\lambda,\Phi)$, we now see that $I=J(\lambda,\Phi)$, and $S/I$ is finite dimensional.\qed
\end{proof}

\bigskip

The following proposition, together with Proposition \ref{kernel-is-cofinite}, is the crucial ingredient to showing that simple admissible modules are always evaluation representations.  We will now add (and use) the hypothesis that $V$ is a simple {\em admissible} (weight) module of $\cL$.

\begin{proposition}\label{radical-ideal}
Let $V$ be a simple admissible module of $\cL$.  Then $I=I(V)$ is a radical ideal of $S$.
\end{proposition}
\begin{proof}
After a few preliminary observations, we will divide the proof into several steps.

Suppose $I$ is {\em not} a radical ideal, and let $A=S/I$.  Then $V$ is a faithful representation of $\fg\ot A$, and $A$ is finite-dimensional by Proposition \ref{kernel-is-cofinite}.  Since $A$ is artinian, its nilradical $\hbox{nilrad}\,A=\sqrt{I}/I$ is nilpotent.  There is thus some $m>1$ such that $(\hbox{nilrad}\,A)^{m-1}\neq 0$, but $(\hbox{nilrad}\,A)^m=0$.  Let $N=(\hbox{nilrad}\,A)^{\ulcorner m/2\urcorner}$, where $\ulcorner m/2\urcorner$ is the least positive integer greater than or equal to $m/2$.  Then $N\neq 0$, but $N^2=0$.

The finite-dimensional abelian subalgebra $\fh\ot N$ commutes with $\fh\ot 1$, so it preserves the (finite-dimensional) weight spaces of $V$.  In particular, $\fh\ot N$ has a nonzero common eigenvector in each weight space of $V$.  Fix such an eigenvector $v\in V$.  

\medskip

\noindent
{\bf \em Step 1}\quad \ {\em Let $\ga\in\Phi^+$ and $s\in N$, and set $x=h_\alpha\ot s$, $y=x_\ga\ot 1$, and $z=2x_\ga\ot s$.  Suppose $z^nv\neq 0$ for some $n\geq 0$.  Then the set $\{z^{n-r}y^rv\ |\ r=0,\ldots,n\}$ is linearly independent.}

\smallskip

\noindent
{\bf\em Proof}\quad  Since $[x,z]\in\fg\ot N^2=0$, we see that $z$ commutes with $x$ and $y$.  By induction, it is easy to verify that 
$$xy^r=y^rx+ry^{r-1}z,$$
in the universal enveloping algebra $U(\cL)$, for $r\geq 1$.  Suppose $R=\sum_{r=0}^nc_rz^{n-r}y^rv=0$ for some $c_r\in k$.  Then\begin{align*}
0=&xR\\
=&\sum_{r=0}^nc_rz^{n-r}(y^rx+ry^{r-1}z)v\\
=&\lambda R+\sum_{r=1}^nc_rrz^{n-r+1}y^{r-1}v,
\end{align*} 
where $\lambda$ is the eigenvalue by which $x$ acts on $v$.  Since $R=0$,
$$\sum_{r=1}^nrc_rz^{n-r+1}y^{r-1}v=0.$$
Applying $x$ again, we obtain
$$\sum_{r=2}^nr(r-1)c_rz^{n-r+2}y^{r-2}v=0,$$
and by induction,
$$\sum_{r=\ell}^nr(r-1)\cdots(r-(\ell-1))c_rz^{n-r+\ell}y^{r-\ell}v=0,$$
for $\ell=0,\ldots,n$.  Taking $\ell=n$, we obtain $n!c_nz^nv=0$, so since $z^nv\neq 0$, we have that $c_n=0$.  Taking $\ell=n-1$ then gives
\begin{align*}
0&=\sum_{r=n-1}^nr(r-1)\cdots(r-n+2)c_rz^{2n-r-1}y^{r-n+1}v\\
&=(n-1)!c_{n-1}z^nv,
\end{align*}
so $c_{n-1}=0$.  By induction on $\ell$, we see that $c_\ell=0$ for $\ell=n,n-1,\ldots,0,$ and the set $\{z^{n-r}y^rv\ |\ r=0,\ldots,n\}$ is linearly independent.

\medskip

\noindent
{\bf\em Step 2}\quad  {\em The element $z=2x_\ga\ot s$ acts nilpotently on $v$.}

\smallskip

\noindent
{\bf\em Proof}\quad  If $z^nv\neq 0$, for all $n$, then by Step 1, $\{z^{n-r}y^rv\ |\ r=0,\ldots,n\}$ is a linearly independent subset of $V_{\mu+n\alpha}$, where $\mu$ is the weight of $v$.  Hence the weight multiplicities of $V$ are unbounded, contradicting the assumption that $V$ is admissible.

\medskip

\noindent
{\bf\em Step 3}\quad  {\em There exists a nonzero weight vector $u\in U(\mathfrak{n}_+\ot N)v$ such that $(\mathfrak{n}_+\ot N)u=0$.}

\smallskip

\noindent
{\bf\em Proof}\quad  By Step 2, every element of the finite-dimensional subalgebra $\mathfrak{n}_+\ot N\subseteq\cL$ acts nilpotently on $v$.  Since $\mathfrak{n}_+\ot N$ is also abelian, it now follows that $U(\mathfrak{n}_+\ot N)v$ is a finite-dimensional module on which the  Lie algebra $\mathfrak{n}_+\ot N$ acts by nilpotent transformations.  By Engel's Theorem, there is a nonzero vector $u\in
U(\mathfrak{n}_+\ot N)v$ for which $(\mathfrak{n}_+\ot N)u=0$.  Since $U(\mathfrak{n}_+\ot N)v$ is weight-graded, we can take $u$ to be a weight vector.
\medskip

\noindent
{\bf\em Step 4}\quad {\em There is a nonzero weight vector $w\in V$, such that
\begin{enumerate}
\item[{\rm (a)}] $((\mathfrak{n}_-\oplus\mathfrak{n}_+)\ot N)w=0$, and
\item[{\rm (b)}] $w$ is a common eigenvector for $\fh\ot N$.
\end{enumerate}  
}
\smallskip

\noindent
{\bf\em Proof}\quad  If we replace $x=h_\ga\ot s$, $y=x_\ga\ot 1$, and $z=2x_\ga\ot s$ by
$$x'=h_\ga\ot s,\ y'=x_{-\ga}\ot 1,\ z'=-2x_{-\ga}\ot s,$$
the arguments given in Steps 1 and 2 show that each element of $\mathfrak{n}_-\ot N$ acts nilpotently on $u$.  As in the proof of Step 3, it now follows that there is a nonzero weight vector $w\in U(\mathfrak{n}_-\ot N)u$ which is annihilated by $\mathfrak{n}_-\ot N$.  Since $\mathfrak{n}_+\ot N$ commutes with $\mathfrak{n}_-\ot N$, $w$ is also annihilated by $\mathfrak{n}_+\ot N$.  

Finally, we note that $w\in U((\mathfrak{n}_-\oplus\mathfrak{n}_+)\ot N)v$ and $\fh\ot N$ commutes with $ U(\mathfrak{n}_-\oplus\mathfrak{n}_+)\ot N)$, so $w$ is a common eigenvector for $\fh\ot N$.

\medskip

\noindent
{\bf\em Step 5}\quad {\em There is a nonzero vector $w\in V$ such that $(\fg\ot N)w=0$.} 

\smallskip

\noindent
{\bf\em Proof}\quad  We need only show that the vector $w$ of Step 4 is annihilated by $\fh\ot N$.  Let $\ga\in\Phi^+$ and $t\in N$, and set $h=h_\ga$, $e=x_\ga$, and $f=x_{-\ga}$.  By Step 4, $(h\ot t)w=\rho w$ for some $\rho\in k$.  We will show that $\rho=0$.

Since $V$ is a weight module, the weight space of $w$ is finite dimensional, so the set $\{(e\ot 1)^r(f\ot 1)^rw\ |\ r\geq 0\}$ is linearly dependent.  Let
$$\sum_{r=0}^\ell c_r(e\ot 1)^r(f\ot 1)^rw=0$$
be a dependence relation with $c_\ell\neq 0$.  By induction, it is straightforward to verify that
\begin{align*}
(e\ot t)(f\ot 1)^r&=(f\ot 1)^r(e\ot t)+r(f\ot 1)^{r-1}(h\ot t)-2\binom{r}{2}(f\ot 1)^{r-2}(f\ot t),\\
(f\ot t)(e\ot 1)^r&=(e\ot 1)^r(f\ot t)-r(e\ot 1)^{r-1}(h\ot t)-2\binom{r}{2}(e\ot 1)^{r-2}(e\ot t).
\end{align*}
Using the fact that $(e\ot t)w=(f\ot t)w=0$, we see that
\begin{align*}
0&=(e\ot t)^\ell(f\ot t)^\ell\sum_{r=0}^\ell c_r(e\ot 1)^r(f\ot 1)^rw\\
&=(-1)^\ell c_\ell(\rho^\ell \ell!)^2w,
\end{align*}
so $\rho=0$ and $(\fh\ot N)w=0$.  Hence $(\fg\ot N)w=0$.

\medskip

\noindent
{\bf\em Step 6}\quad  $I$ is a radical ideal.

\smallskip

\noindent
{\bf\em Proof}\quad  Let $W=\{a\in V\ |\ (\fg\ot N)a=0\}$.  Since $N$ is an ideal of $S$, $W$ is an $\cL$-submodule of $V$.  But $W$ is nonzero by Step 5, so $W=V$ since $V$ is simple.  Therefore, $(\fg\ot N)V=0$, contradicting the faithfulness of the action of $\fg\ot A$ on $V$.  Hence $I$ is a radical ideal of $S$.\qed
\end{proof}

\bigskip

The fact that $I$ is a radical ideal lets us prove the following theorem, one of the main results of this paper.

\begin{theorem}\label{modules-are-evaluation}
Let $\phi:\ \cL\rightarrow\hbox{End}\,V$ be an irreducible admissible representation.  Then $V$ is isomorphic to an evaluation module.
\end{theorem}
\begin{proof}
By Theorem \ref{thm:evaluations}, it is enough to show that $\phi$ factors through an evaluation map $\ev_{\underline{M}}:\ \cL\rightarrow\fg^{\oplus r}$ for some collection of distinct maximal ideals $\underline{M}=(M_1,\ldots,M_r)$ of $S$.

By Propositions \ref{kernel-is-cofinite} and \ref{radical-ideal}, the kernel $\ker\phi=\fg\ot I$ is cofinite in $\cL$, and $I$ is a radical ideal.  It follows that $I$ is the intersection of finitely many distinct maximal ideals $M_1,\ldots,M_r\in\Max\,S$.  By the Chinese Remainder Theorem, 
$$\cL/\ker\phi\cong\fg\ot(S/M_1\oplus\cdots\oplus S/M_r),$$
and since $S$ is of finite type over an algebraically closed field $k$, $S/M_i\cong k$ for all $i$.  That is, 
$$\cL/\ker\phi\cong\fg\ot(S/I)=\fg\ot (S/\cap_{i=1}^r M_i)\cong\fg\ot(S/M_1\oplus\cdots\oplus S/M_r)\cong\fg^{\oplus r},$$
and the composite map $\cL\rightarrow\cL/\ker\phi\stackrel{\cong}{\rightarrow}\fg^{\oplus r}$ is precisely $\ev_{\underline{M}}$.  Therefore, $\phi$ factors through $\ev_{\underline{M}}$, and $V$ is isomorphic to an evaluation module.\qed
\end{proof}

\bigskip

By combining Theorem \ref{labelling-of-evalreps}, Theorem \ref{3.14}, and Theorem \ref{modules-are-evaluation}, we obtain our main theorem, the following classification of the simple admissible modules for current algebras:

\begin{theorem}\label{classification}

\begin{enumerate}
\item[{\rm(1)}] Let $V$ be a simple admissible $\cL$-module.  Then there is a collection of distinct maximal ideals $M_1,\ldots,M_r\subset S$ and simple admissible $\fg$-modules $W_1,\ldots,W_r$ such that at most one $W_i$ is infinite dimensional, and $V\cong V(\underline{M},\underline{W})$, where $\underline{M}=(M_1,\ldots,M_r)$ and $\underline{W}=(W_1,\ldots ,W_r)$.
\item[{\rm(2)}] If $\underline{M}=(M_1,\ldots,M_r)$ consists of distinct maximal ideals of $S$, and each $W_i$ in $\underline{W}=(W_1,\ldots ,W_r)$ is simple, then $V(\underline{M},\underline{W})$ is a simple $\cL$-module.  Such an $\cL$-module is also admissible if and only if each $W_i$ is admissible and at most one $W_i$ is infinite dimensional.

\item[{\rm(3)}] Let $\mathcal{A}$ be the set of all isomorphism classes of simple admissible $\fg$-modules, and let $\mathcal{A^\infty}\subseteq\mathcal{A}$ be the set of isomorphism classes of infinite-dimensional simple admissible $\fg$-modules.  The isomorphism classes of simple admissible $\cL$-modules are in natural bijection with the finitely supported functions $\Psi:\Max\,S\rightarrow\mathcal{A}$ for which $\Psi^{-1}(\mathcal{A}^\infty)$ has cardinality at most $1$.
\end{enumerate}
\end{theorem}
\proof This theorem is an immediate consequence of Theorems \ref{labelling-of-evalreps}, \ref{3.14}, and \ref{modules-are-evaluation}.\qed

\begin{remark} As shown in \cite{BBL}, infinite-dimensional simple admissible $\fg$-modules exist only when $\fg$ is of type $A$ or $C$.  The simple admissible $\cL$-modules are thus finite dimensional for all other $\fg$.
\end{remark}

\begin{remark}  
By Theorem \ref{centre-acts-as-zero}, the statement of Theorem \ref{classification} also holds when $\cL$ is replaced by its universal central extension $\widetilde{\cL}$.
\end{remark}

\end{document}